%% file: sn-article.tex
\documentclass[sn-mathphys,Numbered]{sn-jnl}

\usepackage{graphicx}%
\usepackage{multirow}%
\usepackage{amsmath,amssymb,amsfonts}%
\usepackage{amsthm}%
\usepackage{mathrsfs}%
\usepackage[title]{appendix}%
\usepackage{xcolor}%
\usepackage{textcomp}%
\usepackage{manyfoot}%
\usepackage{booktabs}%
\usepackage{algorithm}%
\usepackage{algorithmic}
\usepackage{listings}%

\newcommand{\argmin}{\operatornamewithlimits{argmin}}

\newtheorem{theorem}{Theorem}
\newtheorem{lemma}{Lemma}
\newtheorem{assumption}{Assumption}

\newtheorem{assumptionalt}{Assumption}[assumption]

\newenvironment{assumptionp}[1]{
  \renewcommand\theassumptionalt{#1}
  \assumptionalt
}{\endassumptionalt}

\newtheorem{definition}{Definition}%

\begin{document}

\title[Goldstein Stationarity under Constraints]{Goldstein Stationarity in Lipschitz Constrained Optimization}


\author[1]{\fnm{Benjamin} \sur{Grimmer}}\email{grimmer@jhu.edu}
\equalcont{These authors contributed equally to this work.}

\author[2]{\fnm{Zhichao} \sur{Jia}}\email{zjia75@gatech.edu}
\equalcont{These authors contributed equally to this work.}

\affil[1]{\orgdiv{Department of Applied Mathematics and Statistics}, \orgname{Johns Hopkins University}, \orgaddress{\city{Baltimore}, \state{MD}, \country{USA}}}

\affil[2]{\orgdiv{H. Milton Stewart School of Industrial and Systems Engineering}, \orgname{Georgia Institute of Technology}, \orgaddress{\city{Atlanta}, \state{GA}, \country{USA}}}


\abstract{We prove the first convergence guarantees for a subgradient method minimizing a generic Lipschitz function over generic Lipschitz inequality constraints. No smoothness or convexity (or weak convexity) assumptions are made. Instead, we utilize a sequence of recent advances in Lipschitz unconstrained minimization, which showed convergence rates of $O(1/\delta\epsilon^3)$ towards reaching a ``Goldstein'' stationary point, that is, a point where an average of gradients sampled at most distance $\delta$ away has size at most $\epsilon$. We generalize these prior techniques to handle functional constraints, proposing a subgradient-type method with similar $O(1/\delta\epsilon^3)$ guarantees on reaching a Goldstein Fritz-John or Goldstein KKT stationary point, depending on whether a certain Goldstein-style generalization of constraint qualification holds.}




\maketitle

\input{introduction}
\input{results}
\input{proofs}


\noindent {\bf Statements and Declarations.} Benjamin Grimmer was supported by the Air Force Office of Scientific Research under award number FA9550-23-1-0531. No data sets were used in or generated by this work. The authors have no competing interests to declare.

\bibliography{sn-bibliography}

\end{document}

%% file: introduction.tex
\section{Introduction}
This work considers a quite general family of constrained optimization problems where both the objective and constraint functions may be nonconvex and nonsmooth. The only structure assumed is Lipschitz continuity of these functions. Specifically, we consider problems of the following form
\begin{equation}
p_\star = \begin{cases}
    \min_{x \in \mathbb{R}^n} \quad &f(x) \\
    \mathrm{s.t.} \quad & g_i(x) \leq 0, \qquad i=1,...,m 
    \label{originalmainproblem}
\end{cases}
\end{equation}
with objective $f:\mathbb{R}^n \to \mathbb{R}$ and constraints $g_i:\mathbb{R}^n \to \mathbb{R}$. We assume these functions are $M$-Lipschitz continuous on a neighborhood of the feasible region $\{x \mid g_i(x)\leq 0\}$, but do not assume convexity or differentiability. 
We propose an iterative subgradient-type method that only computes one gradient (or subgradient-like vector) of one of these functions at each iteration. For now, we leave general the exact notion of our first-order oracle $\mathcal{G}_h(x)$ and associated subdifferential by $\partial h(x)$. Throughout, this will either be the Clarke subdifferential or a certain nonstandard directional subdifferential.

For unconstrained minimization (i.e., $m=0$), in 1977, Goldstein~\cite{goldstein1977optimization} proposed an idealized algorithm achieving descent by repeatedly moving in minimal norm subgradient directions, eventually reaching a related notion of approximate stationarity. Mahdavi-Amiri and Yousefpou~\cite{mahdavi2012effective} revisited this idea, giving a more practical minimization algorithm approximating these descent directions. {\color{black} Curtis and Overton~\cite{curtis2012sequential} proposed a sequential quadratic programming method based on seeking Goldstein descent directions, providing limiting convergence guarantees, although no convergence rates.} The groundbreaking work of Zhang et al.~\cite{zhang2020complexity} showed assuming only Lipschitz continuity is sufficient for provable convergence rates towards stationarity. Subsequently, Davis et al.~\cite{davis2022gradient} and Kong and Lewis~\cite{kong2023cost} built on this using the more standard Clarke subgradient oracle and avoiding the use of randomness, respectively. A key insight enabling these works to prove guarantees for such generic problems is considering a weakened notion of stationarity, seeking points where a convex combination of gradients computed nearby is small. Formally, we say $x$ is an {\color{black} $(\delta,\epsilon)$-Goldstein stationary point} of $h$ if
$ \mathrm{dist}(0,\partial_\delta h(x))\leq \epsilon$ where $\partial_\delta h(x)$ denotes the Goldstein subdifferential defined as
$${\color{black} \partial_\delta h(x) = \mathrm{conv}\left\{ \bigcup_{y\in B(x,\delta)} \partial h(y)\right\} }$$
where $B(x,\delta)$ is the closed ball of radius $\delta$ around $x$ and $\mathrm{conv}()$ denotes the convex hull.
This condition is a natural relaxation of the optimality condition $0\in\partial h(x)$.

Davis et al.~\cite{davis2022gradient} showed for any Lipschitz $h$, using $O(1/\delta\epsilon^3)$ gradient evaluations, an $(\epsilon,\delta)$-Goldstein stationary point can be found in expectation (where the randomness arises from sampling a random point for each gradient evaluation). Kong and Lewis~\cite{kong2023cost} showed an entirely deterministic method could find such a point in $O(1/\delta\epsilon^4)$ subgradient-like evaluations (although depending on a certain nonconvexity modulus). Cutkosky et al.~\cite{CutkoskyMO23} derived a $O(1/\delta\epsilon^3)$ rate using only stochastic gradient evaluations (beyond the scope of this work) and matching lower bounds when $M=\Omega(\sqrt{\frac{\epsilon}{\delta}})$. Both Davis et al.~and Kong and Lewis use a double-loop method, only differing in their inner loop's implementation, which approximates Goldstein's descent direction.

This work generalizes these prior works to the functionally constrained setting~\eqref{originalmainproblem}. We show that a simple modification of previous double-loop methods can solve generic constrained problems of the form~\eqref{originalmainproblem}. To analyze this, we propose new Goldstein-type generalizations of classic Fritz-John, KKT, and Constraint Qualification conditions (see Definitions~\ref{def_gfj}-\ref{def_gcq}). Our Theorems~\ref{theorem_FJ} and~\ref{theorem_KKT} show that using the same inner-loop as Davis et al.~\cite{davis2022gradient}, an approximate Goldstein Fritz John or KKT point is reached at the same rate, depending on whether constraint qualification holds, in expectation. Similarly, using the inner loop of Kong and Lewis~\cite{kong2023cost} extends their deterministic rate to functional constraints. Associated dual multipliers can be extracted from our method, so stopping criteria and certificates of approximate stationarity follow immediately. 

\subsection{Related Works and a Sketch of our Algorithm}
We first review classic results on Lagrangian optimality conditions and provably good methods for weakly convex constrained problems and for Lipschitz unconstrained problems. Subsequently, Section~\ref{section_algorithm&convergence} formalizes all the necessary definitions and our algorithm. Section~\ref{section_analysis} then states and proves our convergence theorems.

\noindent {\bf Lagrangians, KKT, and Constraint Qualification.} 
The Lagrangian of~\eqref{originalmainproblem} for any nonnegative Lagrange multipliers $\lambda\in\mathbb{R}^m$ is denoted by $L(x, \lambda)=f(x)+ \sum_{i=1}^m \lambda_i g_i(x)$.
We say a feasible point $x^\star\in\mathbb{R}^n$ is KKT stationary if there exist such multipliers where
\begin{equation} \label{eq:KKT}
    \begin{cases} \partial f(x^\star) + \sum_{i=1}^m\lambda_i \partial g_i(x^\star) \ni 0 \ ,\\
    \lambda_i g_i(x^\star)=0, \quad \forall i=1\dots m \ . \end{cases}
\end{equation}
Generally, this condition is not necessary for $x^\star$ to be a minimizer of~\eqref{originalmainproblem}. Instead, one can only ensure every minimizer $x^\star$ is Fritz-John (FJ) stationary, meaning that there exist nonnegative multipliers $\gamma_0,\dots \gamma_m$, not all zero, such that
\begin{equation} \label{eq:FJ}
    \begin{cases}\gamma_0\partial f(x^\star) + \sum_{i=1}^m\gamma_i \partial g_i(x^\star) \ni 0 \ ,\\
    \gamma_i g_i(x^\star)=0, \quad \forall i=1\dots m \ . \end{cases}
\end{equation}
The above FJ condition implies the KKT condition whenever the ``constraint qualification'' (CQ) condition $0\not\in \partial \max\{g_i(x) \mid g_i(x^*)=0 \}(x^\star)$ holds at $x^\star$.\footnote{In particular, constraint qualification ensures $\gamma_0$ is positive. Hence $\lambda_i = \gamma_i/\gamma_0$ satisfies KKT.} Hence, when CQ holds, KKT is a necessary optimality condition. Our theory gives guarantees on the approximate attainment of FJ or KKT conditions, depending on whether a slightly strengthened constraint qualification condition holds. 

\noindent {\bf Prior {\color{black} Convergence Rates} for Weakly Convex Constrained Problems.} 
Previous works~\cite{ma2020quadratically,boob2023stochastic,jia2022first} have addressed nonsmooth nonconvex constrained problems under the additional assumption that $f,g_0,\dots,g_m$ are all $\rho$-weakly convex. A function $h$ is $\rho$-weakly convex if $h + \frac{\rho}{2}\|\cdot\|^2$ is convex. Under this additional condition, one can achieve an objective decrease on~\eqref{originalmainproblem} by solving the following strongly convex proximal subproblem
\begin{align}\label{eq:proximal-approach}
\begin{cases}
    \min_{x \in \mathbb{R}^n} \quad &f(x)+\frac{\hat{\rho}}{2}\|x-x_k\|^2 \\
    \mathrm{s.t.} \quad &g_i(x)+\frac{\hat{\rho}}{2}\|x-x_k\|^2 \leq 0, \qquad i=1,...,m
\end{cases}
\end{align}
for any $\hat\rho > \rho$. The works~\cite{ma2020quadratically,boob2023stochastic,jia2022first} all considered double-loop methods where the inner loop approximately solves this subproblem to produce the next iterate of the outer loop. In all these works, convergence rates on the order of $O(1/\epsilon^4)$ were proven towards approximate KKT stationarity holding near the iterates. Additionally,~\cite{jia2022first} showed that without constraint qualification, approximate FJ stationarity is still reached at the same rate. A direct, single-loop approach was recently developed by~\cite{huang2024oracle}, attaining the same rates via a much simpler iteration.

\noindent {\bf Prior {\color{black} Convergence Rates} for Lipschitz Unconstrained Problems.} 
Inspired by Goldstein~\cite{goldstein1977optimization}, previous works~\cite{mahdavi2012effective,zhang2020complexity,davis2022gradient,kong2023cost,CutkoskyMO23} have proposed their algorithms for solving nonsmooth nonconvex unconstrained problems under Lipschitz continuity. These works considered double-loop methods, where the inner loop computes an approximate minimal norm element of the Goldstein subdifferential, either being close to zero or producing at least some descent on the objective function. These descent directions are then followed in an outer loop until a small Goldstein subgradient is found. 

\noindent {\bf Sketch of our Method for Lipschitz, Constrained Problems.} 
Without weak convexity, the proximal subproblem~\eqref{eq:proximal-approach} becomes intractable. Instead, at each outer iteration $k$, we consider the unconstrained Lipschitz subproblem of minimizing $h_{x_k}(x) := \max\{f(x)-f(x_k), {\color{black}g_1(x),...,g_m(x)}\}$.
Using a subgradient norm minimizing inner loop, we either identify a descent on $h_{x_k}$ from $x_k$, yielding a feasible descent for~\eqref{originalmainproblem}, or certify that $x_k$ is approximately Goldstein stationary to $h_{x_k}$, yielding approximate FJ or KKT stationarity for~\eqref{originalmainproblem}.

%% file: results.tex
\section{Goldstein KKT Stationarity and our Algorithm}
\label{section_algorithm&convergence}

For a fixed notion of subdifferential $\partial h(x)$ (two such models are formalized in Sections~\ref{subsec:Damek} and~\ref{subsec:Lewis}), we define Goldstein-type measures of approximate notions of Fritz-John~\eqref{eq:FJ} and KKT stationarities~\eqref{eq:KKT} for the constrained problem~\eqref{originalmainproblem}.

\begin{definition}
    We say a feasible point $x$ is $(\delta,\epsilon,\eta)$-Goldstein Fritz-John (GFJ) stationary for problem \eqref{originalmainproblem} if there exists $(\gamma_0,\gamma)\geq 0$ with $\gamma_0+\sum_{i=1}^m\gamma_i=1$, such that $\mathrm{dist}(0, \gamma_0\partial_\delta f(x)+\sum_{i=1}^m\gamma_i\partial_\delta g_i(x)) \leq \epsilon$ and $\max_{y \in B(x,\delta)}|\gamma_ig_i(y)| \leq \eta, \forall i=1,...,m$.
    \label{def_gfj}
\end{definition}

\begin{definition}
    We say a feasible point $x$ is $(\delta,\epsilon,\eta)$-Goldstein KKT (GKKT) stationary for problem \eqref{originalmainproblem} if there exists $\lambda\geq 0$ such that $\mathrm{dist}(0, \partial_\delta f(x)+\sum_{i=1}^m\lambda_i\partial_\delta g_i(x)) \leq \epsilon$ and $\max_{y \in B(x,\delta)}|\lambda_ig_i(y)| \leq \eta, \forall i=1,...,m$.
    \label{def_gkkt}
\end{definition}

GFJ stationarity implies GKKT stationarity when $\gamma_0>0$ in Definition~\ref{def_gfj}. A larger value of $\gamma_0$ leads to stronger guarantees on approximate GKKT stationarity. In particular, GFJ stationarity is equivalent to GKKT stationarity if $\gamma_0=1$. Constraint Qualification (CQ) classically provides a means to bound $\gamma_0$ away from zero. Below, we generalize CQ to use Goldstein subdifferentials.

{\color{black}
\begin{definition}
    We say a feasible point $x$ satisfies $(a, b, c)$-Goldstein Constraint Qualification (GCQ) if $\mathrm{dist}(0, \sum_{i=1}^m\gamma_i\partial_ag_{i}(x)) \geq b$ holds for all $\gamma \geq 0$ satisfying $\sum_{i=1}^m\gamma_i=1$ and having $\gamma_i=0$ if $g_i(x)<-c$.
    \label{def_gcq}
\end{definition}

GCQ requires a uniform lower bound $b$ for the size of any convex combination of the nearby subgradients (within distance $a$) of constraint functions that are nearly active ($g_i(x)\geq -c$). We compare GCQ with MFCQ (a type of CQ, see~\cite{jia2022first} for its generalization to nonsmooth problems). With $a=0$, $b=\sigma$ and $c=0$, GCQ is equivalent to the ``$\sigma$-strong MFCQ'' defined in~\cite{jia2022first}, which requires a uniform lower bound $\sigma$ for the size of any convex combination of the subgradients of all the active constraint functions at any feasible $x$. GCQ is stronger than MFCQ under positive $a$ and $c$.\footnote{\color{black} We illustrate the relationship between GCQ and MFCQ through a pair of simple examples. Consider a single constraint $g_1(x):=x^2-1\leq 0$ for $x\in\mathbb{R}$, at any $x$, $\sigma$-strong MFCQ holds with $\sigma=2$, while $(a,b,c)$-GCQ holds with any $c\in(0,1)$, $a\in\left(0,\sqrt{1-c}\right)$ and $b=2\left(\sqrt{1-c}-a\right)<2$. If we add another constraint $g_2(x)$ that is defined as $|x|-1$ if $x\in[-1.5,1.5]$ and $0.5$ otherwise, then at any $x$, $\sigma$-strong MFCQ holds with $\sigma=1$, while $(a,b,c)$-GCQ holds with any $c\in(0,1)$, $a\in\left(0,\min\{1-c,0.5\}\right)$ and $b=\min\left\{2\left(\sqrt{1-c}-a\right),1\right\}\leq 1$.}}

\subsection{The Proposed Constrained Goldstein Subgradient Method}
We make the following pair of modeling assumptions throughout.

\begin{assumption}
    The optimal objective value $p_\star$ is finite and for some $\Delta>0$, $f$ and each $g_i$ are $M$-Lipschitz continuous within distance $\Delta$ of $\{x \mid g_i(x)\leq 0\}$.\label{assumption1}
\end{assumption}
    
\begin{assumption}
An initial feasible point $x_0$ is known (i.e. $g(x_0) \leq 0$).\label{assumption2}
\end{assumption}
Without loss of generality, the $m$ constraints of~\eqref{mainproblem} can be combined into one:
\begin{equation}
p_\star = \begin{cases}
     \min_{x \in \mathbb{R}^n} \quad &f(x) \\
    \mathrm{s.t.} \quad &g(x):=\max_{i=1,...,m}g_i(x) \leq 0 \ .
    \label{mainproblem}
\end{cases}
\end{equation}
Our proposed subgradient method then maintains a feasible solution $x_k$ and proceeds by, in an inner loop, computing in a descent direction of
\begin{align} \label{eq:subproblem}
    \min h_{x_k}(z) := \max\{f(z)-f(x_k), g(z)\} \ ,
\end{align}
and then moving in this direction a fixed distance $\delta<\Delta$ in an outer loop. To facilitate this, we require some subgradient-type oracle $\mathcal{G}$ be provided for $f$, $g$, and $h_x$. The two particular computational models we consider here are formalized in Assumptions~\ref{assumption3a} and~\ref{assumption3b}, with $\mathcal{G}$ as a gradient almost everywhere or as a specialized directional derivative matching subgradient, respectively. In general, all we require is the following.

\begin{assumption}
For any $x,z \in \mathbb{R}^n$, subgradient-type oracles are known with $\mathcal{G}_f(z) \in \partial f(z), \mathcal{G}_g(z) \in \partial g(z), \mathcal{G}_{h_x}(z) \in \partial h_x(z)$ satisfying
$$\mathcal{G}_{h_{x}}(z) \in \begin{cases}
\{\mathcal{G}_f(z)\} & \text{ if } f(z)-f(x) > g(z)\\
\{\mathcal{G}_f(z), \mathcal{G}_g(z)\} & \text{ if } f(z)-f(x) = g(z)\\
\{ \mathcal{G}_g(z)\} & \text{ if } f(z)-f(x) < g(z)\ . \end{cases}$$ \label{assumption3}
\end{assumption}

We require a subroutine outputting an approximate minimal norm element of the Goldstein subdifferential set of $h_{x_k}$ at $x_k$ in each outer iteration $k$. 
We consider any subroutine $\mathcal{A}(x,h,\delta,\epsilon)$ producing a direction $\zeta\in\mathbb{R}^n$  defined by $t \leq T_{\mathcal{A}}(M,\epsilon,\delta,\tau)$ subgradient oracle calls with probability at least $(1-\tau)$ satisfying
$$\|\zeta\| \leq \epsilon \qquad \mathrm{or} \qquad h(x)-h(x-\delta \zeta / \|\zeta\|) \geq C\delta \epsilon \ ,$$
with $\zeta =\sum_{i=1}^t w_i \mathcal{G}_{h_x}(z_i)$ for some points $z_i \in B(x, \delta)$ and weights $w\geq 0, \sum_{i=1}^t w_i=1$ for some fixed $C \in (0, 1)$ and $\tau \in [0, 1)$.

\begin{algorithm}[t!]
\caption{Constrained Goldstein Subgradient Method}
\label{ouralgorithm}
\begin{algorithmic}
\REQUIRE{{\color{black} $\delta\in(0,\Delta)$, $\Tilde{\epsilon}>0$, $x_0$ with $g(x_0) \leq 0$}}
\FOR{$k=0, 1, ..., $}
\STATE{Compute $\zeta_k=\mathcal{A}(x_k, h_{x_k}, \delta, \Tilde{\epsilon})$, an approximate minimal element of $\partial_\delta h_{x_k}(x_k)$}
\STATE{{\bf if } $\|\zeta_k\| \leq \Tilde{\epsilon}$ {\bf then return } $x_k$}
\STATE{$x_{k+1}=x_k-\delta \zeta_k / \|\zeta_k\|$, \ $k=k+1$}
\ENDFOR
\end{algorithmic}
\end{algorithm}

The evaluation of such a subroutine forms the inner loop of our method.
As two examples, Davis et al.~\cite{davis2022gradient} gave a randomized searching subroutine, which we denote as $RandSearch(x, h, \delta, \epsilon)$, and Kong and Lewis~\cite{kong2023cost} gave a deterministic linesearching subroutine defined below as $BisecSearch(x, h, \delta, \epsilon)$.
Our proposed algorithm can use any such subroutine as a black-box inner loop, formalized in Algorithm~\ref{ouralgorithm}. The following two lemmas show finite termination and a relationship between Goldstein stationarity on $h_{x_k}$ and GFJ and GKKT stationarity (proofs deferred to Section~\ref{section_analysis}).
\begin{lemma}
    Algorithm~\ref{ouralgorithm} always terminates with a feasible $x_k$ with
    $k \leq \left\lceil\frac{f(x_0)-p_\star}{C\delta \epsilon}\right\rceil$.
    \label{lemma_outeriterations}
\end{lemma}
\begin{lemma}
    If a subroutine $\mathcal{A}$ reports a feasible $x$ is $(\delta,\epsilon)$-Goldstein stationary to $h_x$, then $x$ is $(\delta,\epsilon,3M\delta)$-GFJ stationary for problem~\eqref{mainproblem}. If {\color{black} $(\delta, \hat{\epsilon}, 2M\delta)$-GCQ} holds at $x$ with $\hat{\epsilon}>\epsilon$, then $x$ is $(\delta, \epsilon(\hat{\epsilon}+M)/(\hat{\epsilon}-\epsilon),3M\delta(\hat{\epsilon}+M)/(\hat{\epsilon}-\epsilon))$-GKKT stationary.
    \label{lemma_fjtokkt}
\end{lemma}

The Lagrange multiplier $\lambda_k$ associated with each iteration $x_k$ can be easily computed: Let $z_i$ and $w_i$ denote the points and weights associated with the construction of $\zeta_k$. By Assumption~\ref{assumption3}, letting $I_f =\{ i \mid \mathcal{G}_f(z_i)=\mathcal{G}_{h_{x_k}}(z_i)\}$, we have $\zeta_k=\sum_{i\in I_f} w_i\mathcal{G}_f(z_i)+\sum_{i\not\in I_f} w_i\mathcal{G}_g(z_i)$. Provided $\sum_{i\in I_f} w_i>0$, which GCQ ensures, $\lambda_k=\sum_{i\not\in I_f}w_i/\sum_{i\in I_f}w_i$ certifies approximate GKKT stationarity.

\subsection{Oracles and Minimum Norm Subroutine of Davis et al.~\texorpdfstring{\cite{davis2022gradient}}{}} \label{subsec:Damek}
Note that the assumed uniform Lipschitz continuity ensures gradients exist almost everywhere for $f$, $g$, and $h_{x_k}$ (by Rademacher). Davis et al.~\cite{davis2022gradient} leverage this by only requiring subgradients at randomly sampled points that are differentiable with probability one. Such an oracle can often be produced via automatic differentiation.
\begin{assumptionp}{\text{3a}}
    Oracles computing $\mathcal{G}_f(x) = \nabla f(x)$ and $\mathcal{G}_g(x) = \nabla g(x)$ almost everywhere in $\mathbb{R}^n$ are known.
    \label{assumption3a}
\end{assumptionp}

Given these oracles, a gradient oracle for $h_{x_k}$ almost everywhere can be easily constructed, returning $\nabla f(x)$ if $f(x)-f(x_k) \geq g(x)$ and $\nabla g(x)$ otherwise. Using this oracle, the randomized search subroutine of~\cite{davis2022gradient} finds an approximate minimum norm Goldstein subgradient as defined in Algorithm~\ref{davisalgorithm}. The natural corresponding subdifferential is the set of all Clarke subgradients, given by $\partial h(x) = \mathrm{conv}\left\{ \lim_{i\rightarrow\infty} \nabla h(x_i) \mid x_i\rightarrow x,\ x_i\in\mathrm{dom}(\nabla h) \right\}$.

\subsection{Oracles and Minimum Norm Subroutine of Kong et al.~\texorpdfstring{\cite{kong2023cost}}{}} \label{subsec:Lewis}
Alternatively, Kong and Lewis~\cite{kong2023cost} showed an entirely deterministic approach could be used at the cost of requiring additional mild assumptions and a stronger subgradient-type oracle.
The directional derivative of any function $h:\mathbb{R}^n \to \mathbb{R}$ is defined by
\begin{align*}
    \nabla_v h(x)=\lim_{t \to 0^+}\frac{h(x+tv)-h(x)}{t} \ ,
\end{align*}
for any point $x \in \mathbb{R}^n$ and any direction $v \in \mathbb{R}^n$. Kong and Lewis~\cite{kong2023cost}'s subroutine requires an oracle able to compute subgradient-like vectors in agreement with this directional derivative and bounds on the following quantities. 
For any one-dimensional function $\ell \colon [0,\delta]\rightarrow \mathbb{R}$, we denote the ``concave deviation'' by
$$ c_\delta(\ell ) = \inf\{ M \geq 0 \mid \ell + s \text{ is convex for some convex $M$-Lipschitz } s\colon [0,\delta]\rightarrow \mathbb{R} \} \ . $$
This notion extends to multivariate functions $h\colon \mathbb{R}^n\rightarrow \mathbb{R}$ by considering each restriction of $h$ to a line segment $\ell_{h,x,y}(r)=h(x+r(y-x)/\delta)$ for $\|x-y\|\leq \delta$. We denote the nonconvexity modulus of $h$ by the largest concave deviation among these $$\Lambda_h(\delta) = \sup_{\|x-y\|\leq \delta}\{c_\delta(\ell_{h,x,y})\} \ .$$
\begin{assumptionp}{\text{3b}}
    The functions $f$ and $g$ are directionally differentiable, have oracles producing directional subgradient maps $F(x,v)$, $G(x,v)$ satisfying $\langle F(x,v),v\rangle=\nabla_v f(x)$, $\langle G(x,v),v\rangle=\nabla_v g(x)$, and have $\Lambda_f(\delta)$ and $\Lambda_g(\delta)$ both finite.
    \label{assumption3b}
\end{assumptionp}
\begin{algorithm}[t!]
\caption{{\color{black} $RandSearch(z, h_z, \delta, \epsilon)$}}
\label{davisalgorithm}
\begin{algorithmic}
\REQUIRE{{\color{black} a feasible $z$, $h_z$, $\delta\in(0,\Delta)$, $\epsilon>0$}}
\STATE{Let $t=0$, $\zeta_0=\nabla h_{z}(y_0)$, where $y_0 \sim B(z, \delta)$}
\WHILE{$\|\zeta_t\|>\epsilon$ and $\delta \|\zeta_t\|/4 \geq h_{z}(z)-h_{z}(z-\delta \zeta_t/\|\zeta_t\|)$}
\STATE{Choose any $r$ satisfying $0<r<\|\zeta_t\|\cdot\sqrt{1-(1-\|\zeta_t\|^2/128M^2)^2}$}
\STATE{Sample $y_t$ uniformly from $B(\zeta_t, r)$}
\STATE{Choose $s_t$ uniformly at random from the segment $[z, z-\delta y_t/\|y_t\|]$}
\STATE{$\zeta_{t+1}=\argmin_{\zeta \in [\zeta_t, \nabla h_{z}(s_t)]}\|\zeta\|$, \ $t=t+1$}
\ENDWHILE
\end{algorithmic}
\end{algorithm}
\begin{algorithm}[t!]
\caption{{\color{black} $BisecSearch(z, h_z, \delta, \epsilon)$}}
\label{lewisalgorithm}
\begin{algorithmic}
\REQUIRE{{\color{black} a feasible $z$, $h_z$, $\delta\in(0,\Delta)$, $\epsilon>0$}}
\STATE{Let $t=0$, $\zeta_0=\nabla h_{z}(z)$}
\WHILE{$\|\zeta_t\|>\epsilon$ and $\delta \epsilon/3>h_{z}(z)-h_{z}(z-\delta \hat{\zeta}_t)$}
\STATE{Use {\color{black} bisection~\cite[Algorithm 1]{kong2023cost}} to find $r_t$ satisfying $\langle H_z(z_t(r_t),\hat{\zeta}_t), \hat{\zeta}_t \rangle-\epsilon/2<0$}
\STATE{$\zeta_{t+1}=\argmin_{\zeta \in [\zeta_t, H_z(z_t(r_t), \hat{\zeta}_t)]}\|\zeta\|$, \ $t=t+1$}
\ENDWHILE
\end{algorithmic}
\end{algorithm}

For any $x,z,v\in\mathbb{R}^n$, we define the associated directional subgradient map for $h_x$ as
$$ H_x(z,v) = \begin{cases} F(z,v) \text{ if } f(z)-f(x)>g(z) \text{, or } f(z)-f(x)=g(z) \text{ and } \nabla_v f(z) \geq \nabla_v g(z)\\
G(z,v) \text{ if } f(z)-f(x)<g(z) \text{, or } f(z)-f(x)=g(z) \text{ and } \nabla_v f(z) < \nabla_v g(z) \ . \end{cases}$$
This is justified by noting the directional derivative of $h_{x}$ at $z$ is given by
\begin{align*}
    \nabla_v h_{x}(z)=\begin{cases}
        \nabla_v f(z), &{f(z)-f(x)>g(z)}\\
        \max\{\nabla_v f(z), \nabla_v g(z)\}, &{f(z)-f(x)=g(z)} \\
        \nabla_v g(z), &{f(z)-f(x)<g(z)} \ .
    \end{cases}
\end{align*}
Note $\Lambda_{h_x}(\delta) \leq \Lambda_f(\delta)+\Lambda_g(\delta)<\infty$
since any convex Lipschitz functions $s_f,s_g$ such that $\ell_{f,x,y}+s_f$ and $\ell_{g,x,y}+s_g$ are convex have $\ell_{h_z,x,y}+s_f+s_g$ convex.

Suppressing the dependence on the given direction $v$, we denote $\mathcal{G}_f(z) = F(z,v)$, $\mathcal{G}_g(z) = G(z,v)$, and  $\mathcal{G}_{h_x}(z) = H_x(z,v)$.
One corresponding subdifferential is the set of all directional subgradients, $\partial h(x)=\{\zeta \mid \exists v \in \mathbb{R}^n \ \mathrm{s.t.\ } \langle \zeta,v \rangle=\nabla_v h(x)\}$.

Using such oracles, the bisection search subroutine of~\cite{kong2023cost} finds an approximate minimum norm Goldstein subgradient as defined in Algorithm~\ref{lewisalgorithm}. For a given $z\in\mathbb{R}^n$, this subroutine constructs a sequence of Goldstein subgradients $\zeta_t$ of $h_z$ at $z$, eventually satisfying one of the needed conditions.
Defining $\hat{\zeta}_t=\zeta_t/\|\zeta_t\|$, $z_t(r)=z+(r-\delta)\hat{\zeta}_t$ and $l_t(r)=h_z(z_t(r))-\epsilon r/2$ on $[0,\delta]$, knowing $l_t(0)>l_t(\delta)$ (an average rate of decrease), the critical step in this construction is the {\color{black} bisection~\cite[Algorithm 1]{kong2023cost}}, finding a location where the right derivative $l_{t+}'(r)=\lim_{\Delta r \to 0^+}(l_t(r+\Delta r)-l_t(r))/\Delta r$ of $l_t(r)$ is negative by bisection. Finite nonconvexity modulus guarantees termination.

%% file: proofs.tex
\section{Convergence Guarantees and Analysis}
\label{section_analysis}

We provide two main theorems on FJ and KKT guarantees, using Algorithm~\ref{davisalgorithm} or~\ref{lewisalgorithm}, respectively, as the subroutines. These follow primarily from our Lemmas~\ref{lemma_outeriterations} and~\ref{lemma_fjtokkt}. 
\begin{theorem}    \label{theorem_FJ}
    Given Assumptions~\ref{assumption1} and \ref{assumption2}, for any $\epsilon>0$ and $0<\delta<\Delta$, Algorithm~\ref{ouralgorithm} with $\Tilde{\epsilon}=\epsilon$ will terminate with an $(\delta, \epsilon,3M\delta)$-GFJ stationary point. Under Assumption~\ref{assumption3a} and using Algorithm~\ref{davisalgorithm} as the subroutine, with probability $1-\tau$, this requires at most
    $$\left\lceil\frac{4(f(x_0)-p_\star)}{\delta \epsilon}\right\rceil \left\lceil\frac{64M^2}{\epsilon^2}\right\rceil \left\lceil2\log\left(\frac{4(f(x_0)-p_\star)}{\tau\delta \epsilon}\right)\right\rceil$$
    first-order oracle calls. Under Assumption~\ref{assumption3b} and using Algorithm~\ref{lewisalgorithm}, this becomes
    $$\left\lceil\frac{3(f(x_0)-p_\star)}{\delta \epsilon}\right\rceil \left\lceil\frac{16M^2}{\epsilon^2}\right\rceil \left(1+\left\lfloor\frac{12(\Lambda_f(\delta)+\Lambda_g(\delta))}{\epsilon}\right\rfloor\right) \ . $$
\end{theorem}
\begin{proof}
According to Lemma~\ref{lemma_outeriterations}, we attain a $(\delta,\epsilon)$-Goldstein stationary point $x_k$ to $h_{x_k}$ with $g(x_k)<0$ with
$ k \leq \left\lceil\frac{f(x_0)-p_\star}{C\delta \epsilon}\right\rceil. $
According to Lemma~\ref{lemma_fjtokkt}, $x$ is a $(\delta,\epsilon,3M\delta)$-GFJ stationary solution for problem~\eqref{mainproblem}. By {\color{black} Corollary 5} in~\cite{davis2022gradient}, each call to Algorithm~\ref{davisalgorithm} uses at most $\lceil 64M^2/\epsilon^2\rceil\lceil2\log(1/\tau)\rceil$ gradient evaluations with probability at least $1-\tau$ and $C=1/4$. Likewise, by {\color{black} Theorem 1 and Corollary 1 with $\sigma=\epsilon/6$} in~\cite{kong2023cost}, Algorithm~\ref{lewisalgorithm} requires no more than $\lceil 16M^2/\epsilon^2\rceil (1+\lfloor12(\Lambda_f(\delta)+\Lambda_g(\delta))/\epsilon\rfloor)$ subgradient evaluations with $C=1/3$.
\end{proof}

\begin{theorem}
    Given Assumptions~\ref{assumption1}, \ref{assumption2}, for any $\epsilon>0$ and $0<\delta<\Delta$ such that {\color{black} $(\delta, \sigma, 2M\delta)$-GCQ} holds for all $x$, Algorithm~\ref{ouralgorithm} with $\Tilde{\epsilon}=\sigma\epsilon/(\epsilon+\sigma+M)$ will terminate with an $(\delta,\epsilon,3M\delta(\epsilon+\sigma+M)/\sigma)$-Goldstein KKT point. Under Assumption~\ref{assumption3a} and using Algorithm~\ref{davisalgorithm} as the subroutine, with probability $1-\tau$, this requires at most
    $$\left\lceil\frac{4(f(x_0)-p_\star)}{\delta\Tilde{\epsilon}}\right\rceil \left\lceil\frac{64M^2}{\Tilde{\epsilon}^2}\right\rceil \left\lceil2\log\left(\frac{4(f(x_0)-p_\star)}{\tau\delta \Tilde{\epsilon}}\right)\right\rceil$$
    first-order oracle calls. Under Assumption~\ref{assumption3b} and using Algorithm~\ref{lewisalgorithm}, this becomes
    $$\left\lceil\frac{3(f(x_0)-p_\star)}{\delta\Tilde{\epsilon}}\right\rceil \left\lceil\frac{16M^2}{\Tilde{\epsilon}^2}\right\rceil \left(1+\left\lfloor\frac{12(\Lambda_f(\delta)+\Lambda_g(\delta))}{\Tilde{\epsilon}}\right\rfloor\right) \ .$$
    \label{theorem_KKT}
\end{theorem}
\begin{proof}
By Lemmas~\ref{lemma_outeriterations} and~\ref{lemma_fjtokkt}, our choice $\Tilde{\epsilon}=\frac{\sigma\epsilon}{\epsilon+\sigma+M}<\sigma$ ensures $x_k$ is a feasible $(\delta,\Tilde{\epsilon},3M\Tilde{\epsilon})$-GFJ point. Then {\color{black} $(\delta,\sigma,2M\delta)$-GCQ} implies $x_k$ is $(\delta,\epsilon,3M\delta(\epsilon+\sigma+M)/\sigma)$-GKKT stationary, where $\epsilon$ is derived by
$\epsilon=\Tilde{\epsilon}(\sigma+M)/(\sigma-\Tilde{\epsilon})$,
and $3M\delta(\epsilon+\sigma+M)/\sigma=3M\delta(\sigma+M)/(\sigma-\Tilde{\epsilon})$. Using the same subroutine bounds of~\cite{davis2022gradient,kong2023cost} used in the prior proof of Theorem~\ref{theorem_FJ} finishes the proof.
\end{proof}

\subsection{Proof of Lemma~\ref{lemma_outeriterations}}
We prove this by inductively showing feasibility is maintained and the objective value decreases by at least $C\delta\epsilon$ each iteration until $\|\zeta_k\|\leq \epsilon$. Supposing $\|\zeta_k\|>\epsilon$, our black-box subroutine finds $h_{x_k}(x_k)-h_{x_k}(x_{k+1}) \geq C\delta \epsilon$. Since $h_x(x)=\max\{f(x)-f(x), g(x)\}=0$ whenever $g(x) \leq 0$, we know $\max\{f(x_{k+1})-f(x_k),g(x_{k+1})\} \leq -C\delta \epsilon.$
Thus we have descent $f(x_k)-f(x_{k+1}) \geq C\delta \epsilon$ and maintain feasibility $g(x_{k+1}) \leq -C\delta \epsilon$. Given an initial objective gap of $f(x_0)-p_\star$, this occurs at most $\left\lceil\frac{f(x_0)-p_\star}{C\delta \epsilon}\right\rceil$ times.

\subsection{Proof of Lemma~\ref{lemma_fjtokkt}}
{\color{black} If $x$ is $(\delta,\epsilon)$-Goldstein stationary to $h_x$, there exists $t$ points $z_i \in B(x, \delta)$ with (sub)gradients $\mathcal{G}_{h_x}(z_i) \in \partial h_x(z_i)$ and weights $w > 0$ with $\sum_{i=1}^tw_i=1$ such that $\zeta = \sum_{i=1}^t w_i\mathcal{G}_{h_x}(z_i)$ has $\|\zeta\| \leq \epsilon$.}
By Assumption~\ref{assumption3}, letting $I_f =\{ i \mid \mathcal{G}_f(z_i)=\mathcal{G}_{h_{x}}(z_i)\}$, we have $\zeta=\sum_{i\in I_f} w_i\mathcal{G}_f(z_i)+\sum_{i\not\in I_f} w_i\mathcal{G}_g(z_i)$. This motivates selecting $\gamma_0 = \sum_{i\in I_f} w_i \geq 0$ and $\gamma = \sum_{i\not\in I_f} w_i \geq 0$. Indeed this selection has $\gamma_0+\gamma=1$ and $\mathrm{dist}(0, \gamma_0\partial_\delta f(x)+\gamma\partial_\delta g(x)) \leq \epsilon$.
Lastly, for establishing approximate Goldstein Fritz-John stationarity, we verify complementary slackness. This trivially holds if $\gamma=0$. Otherwise let $i\not\in I_f$ and observe that $|g(z)| \leq 3M\delta$ for any $z\in B(x,\delta)$ since we have
$$ g(z) = g(z) - g(x) + g(x) \leq M\delta $$
using $M$-Lipschitz continuity of $g$ on $\|z-x\|\leq \delta$ and the feasibility of $x$, and we have
$$ g(z) = g(z) - g(z_i) + g(z_i) \geq -3M\delta $$
using $M$-Lipschitz continuity of $g$ on $\|z_i-z\|\leq 2\delta$ and $g(z_i) \geq f(z_i)-f(x) \geq -M\delta$. Since $\gamma\leq 1$, $\max_{z\in B(x,\delta)} |\gamma g(z)|\leq 3M\delta$ and so $x$ is a $(\delta,\epsilon,3M\delta)$-GFJ point.

{\color{black} It directly follows from above that $x$ is a $(\delta,\epsilon,3M\delta)$-GKKT point if $\gamma=0$. Otherwise, assuming $(\delta, \hat{\epsilon}, 2M\delta)$-GCQ holds, we first claim $\gamma_0 \geq \frac{\hat{\epsilon}-\epsilon}{\hat{\epsilon}+M} > 0$ since for any $i\notin I_f$, $g(x)=g(x)-g(z_i)+g(z_i)\geq -2M\delta$ using $M$-Lipschitz continuity of $g$ on $\|z_i-x\|\leq\delta$ and previously derived $g(z_i)\geq -M\delta$, and based on this we have}
$$(1-\gamma_0) \hat{\epsilon}=\gamma \hat{\epsilon} \leq \left\|\sum_{i\not\in I_f} w_i\mathcal{G}_g(z_i)\right\| \leq \left\|\sum_{i=1}^t w_i\mathcal{G}_{h_{x}}(z_i)\right\|+\left\|\sum_{i\in I_f} w_i\mathcal{G}_f(z_i)\right\| \leq \epsilon+\gamma_0 M \ .$$
As a result, consider the Lagrangian multiplier $ 0 \leq \lambda := \gamma/\gamma_0 \leq \frac{\hat{\epsilon} + M}{\hat{\epsilon}-\epsilon} - 1$. Then $(\delta, \epsilon(\hat{\epsilon}+M)/(\hat{\epsilon}-\epsilon),3M\delta(\hat{\epsilon}+M)/(\hat{\epsilon}-\epsilon))$-GKKT stationary for problem~\eqref{mainproblem} follows as 
\begin{align*}
    \mathrm{dist}(0, \partial_\delta f(x)+\lambda \partial_\delta g(x)) \leq \frac{\epsilon}{\gamma_0} \leq \frac{\epsilon(\hat{\epsilon}+M)}{\hat{\epsilon}-\epsilon} \ ,\\
    \max_{z\in B(x,\delta)} |\lambda g(z)| \leq \frac{\max_{z\in B(x,\delta)} |\gamma g(z)|}{\gamma_0} \leq \frac{3M\delta (\hat{\epsilon}+M)}{\hat{\epsilon}-\epsilon} \ .
\end{align*}